\documentclass[10pt,a4paper]{article}
\usepackage{fullpage}
\usepackage{amsfonts,amsmath,amssymb}
\usepackage{amsthm}
\usepackage{graphicx}\usepackage{sectsty}
\theoremstyle{plain}
\newtheorem{theorem}{Theorem}[section]
\newtheorem{lemma}[theorem]{Lemma}

\theoremstyle{definition}
\newtheorem{definition}[theorem]{Definition}
\theoremstyle{remark}

\numberwithin{equation}{section}

\sectionfont{\large}

\newenvironment{compliance}[1][Compliance with Ethical Standards
]{\begin{trivlist} \item[\hskip \labelsep {\bfseries
#1}]}{\end{trivlist}}

\begin{document}
\title{A Proof of Schiffer's Conjecture in Starlike Domain\\ by Far-Field Patterns}
\author{Lung-Hui Chen$^1$}\maketitle\footnotetext[1]{Department of
Mathematics, National Chung Cheng University, 168 University Rd.
Min-Hsiung, Chia-Yi County 621, Taiwan. Email:
mr.lunghuichen@gmail.com;\,lhchen@math.ccu.edu.tw. Fax:
886-5-2720497.}
%ARMA------------------------------------------------------------------------%
\begin{abstract}
We formulate the Schiffer's conjecture in spectral geometry in the context of scattering theory. The problem is equivalent to finding a non-trivial solution in an interior transmission problem. We compare the back-scattering data of the perturbation along all incident angles. The uniqueness of the inverse scattering problem along each incident direction proves the Schiffer's conjecture.
/Cartwright-Levinson theory.
\end{abstract}
\section{Introduction}
In this paper, we study the following inverse spectral problem:
\begin{eqnarray}\label{1.1}
\left\{%
\begin{array}{ll}
    \Delta u+k^2u=0,  & \hbox{ in }D ,\,k^2\in\mathbb{R}^+;\vspace{3pt}\\\vspace{3pt}
    \frac{\partial u}{\partial \nu}=0,& \hbox{ on }\partial D;\\
    u=1,&\hbox{ on }\partial D,
\end{array}%
\right.
\end{eqnarray}
where $\nu$ is the unit outer normal; $D$ is a fixed starlike domain in $\mathbb{R}^3$ containing the origin with Lipschitz boundary $\partial D$. We interpret the model as the plane waves perturbed by the boundary condition which is specified by $D$, and satisfies the Helmholtz equation outside $D$. Let $u$ be a non-trivial eigenfunction with some $k^2\in\mathbb{R}^+$. We want to show that $D$ are actually balls centered at origin.

\par
Here we prove the result as a special case of interior transmission problem \cite{Aktosun,Cakoni2,Chen,Chen2,Chen3,Chen5,Colton4,Colton,Colton3,Colton2,
Colton5,Kirsch86,Kirsch,L,La,Liu,Mc,Rynne}. In interior transmission problems, we look for a frequency so that a perturbed stationary wave behaves like or somewhere like a spherical Bessel function outside the perturbation. In Schiffer's conjecture, we ask if there is a frequency so that a perturbed wave can stay in its initial shape traveling to infinity in constant speed.
We refer to \cite{A,Liu} and the reference there for the connections of interior transmission problem to other questions in mathematical science.
\par
To give a point of view from scattering theory to~(\ref{1.1}),
we take the incident wave field to be the time harmonic  acoustic plane wave of the form $$u^i(x):=e^{ikx\cdot d},$$ $k\in\mathbb{R}^+$, $x\in\mathbb{R}^3$, and $d\in\mathbb{S}^2$ is the incident direction.  The inhomogeneity is defined by the index of refraction $n\in\mathcal{C}^2(\mathbb{R}^3)$ of~(\ref{1.1}), and  the wave propagation is governed by the following equation.
\begin{eqnarray}\label{122}
\left\{%
\begin{array}{ll}
\Delta u(x)+k^2n(x)u(x)=0,\,x\in\mathbb{R}^3;\vspace{4pt}\\\vspace{3pt}
u(x)=u^i(x)+u^s(x),\,x\in\mathbb{R}^3\setminus D; \\
\lim_{|x|\rightarrow\infty}|x|\{\frac{\partial u^s(x)}{\partial |x|}-iku^s(x)\}=0,
\end{array}%
\right.
\end{eqnarray}
in which the third equation is the Sommerfeld's radiation condition.
Particularly, we have the following asymptotic expansion on the scattered wave field \cite{Colton2, Isakov}.
\begin{equation}\label{U}
u^s(x)=\frac{e^{ik|x|}}{|x|}u_\infty(\hat{x};d,k)+O(\frac{1}{|x|^{\frac{3}{2}}}),\,|x|\rightarrow\infty,
\end{equation}
which holds uniformly for all $\hat{x}:=\frac{x}{|x|}$, $x\in\mathbb{R}^3$, and $u_\infty(\hat{x};d,k)$ is known as the scattering amplitude or far-field pattern in the literature \cite{Colton2,Kirsch86}. It has an expansion in spherical harmonics \cite[p.\,35,\,Theorem 2.15]{Colton2}
\begin{equation}
u_\infty(\hat{x};d,k)=\frac{1}{k}\sum_{n=0}^\infty\frac{1}{i^{n+1}}\sum_{m=-n}^na_n^mY_n^m(\hat{x}),
\end{equation}
where we follow the notation in the reference.

\par
Let us start with the Rellich's representation in scattering theory.
We expand the possible solution $u$ of~(\ref{1.1}) in a series of spherical harmonics near infinity by Rellich's lemma \cite[p.\,32, p.\,227]{Colton2}:
\begin{eqnarray}\label{1.2}
u(x;k)=\sum_{l=0}^{\infty}\sum_{m=-l}^{m=l}a_{l,m}(r)Y_l^m(\hat{x}),
\end{eqnarray}
where $r:=|x|$, $r\geq R_0$ with a sufficiently large $R_0$; $\hat{x}=(\theta,\varphi)\in\mathbb{S}^2$.
The summations converge uniformly and absolutely on suitable compact subsets away from $D$.
The spherical harmonics
\begin{equation}\label{S}
Y_l^m(\theta,\varphi):=\sqrt{\frac{2l+1}{4\pi}\frac{(l-|m|)!}{(l+|m|)!}}
P_l^{|m|}(\cos\theta)e^{im\varphi},
\,m=-l,\ldots,l;\,l=0,1,2,\ldots,
\end{equation}
form a complete orthonormal system in $\mathcal{L}^2(\mathbb{S}^2)$, in which
\begin{equation}
P_n^m(t):=(1-t^2)^{m/2}\frac{d^mP_n(t)}{dt^m},\,m=0,1,\ldots,n,
\end{equation}
where the Legendre polynomials $P_n$, $n=0,1,\ldots,$
form a complete orthogonal system in $L^2[-1,1]$. We refer this to \cite[p.\,25]{Colton2}.
By the orthogonality of the spherical harmonics, the family of functions
\begin{eqnarray}\label{1.5}
\{u_{l,m}(x;k)\}_{l,m}:=\{a_{l,m}(r)Y_l^m(\hat{x})\}_{l,m}
\end{eqnarray}
satisfy the first equation in~(\ref{1.1}) independently for each $(l,m)$
in $r\geq R_0$ for sufficiently large $R_0$.

\par
Now we consider the boundary condition given by the second and third equations in~(\ref{1.1}), and then extend the solutions $u_{l,m}(x;k)$
into $r\leq R_0$ as follows.
Let $\hat{x}_0\in\mathbb{S}^2$ be any given incident direction that intersects $\partial D$ at $(\hat{R},\hat{x}_0)\in \mathbb{R}^+\times\mathbb{S}^2$.
For the given $\hat{x}_0$, we impose the differential operator%%%I made a correction here%%%
\begin{equation}\nonumber
\Delta=\frac{1}{r^2}\frac{\partial}{\partial r}
r^2\frac{\partial}{\partial r}+\frac{1}{r^2\sin{\varphi}}\frac{\partial}{\partial \varphi}\sin\varphi\frac{\partial}{\partial \varphi}
+\frac{1}{r^2\sin^2{\varphi}}\frac{\partial^2}{\partial \theta^2}
\end{equation}
on $u_{l,m}(x;k)$ and, accordingly, we have the following ODE:
\begin{eqnarray}
 \frac{d^2 a_{l,m}(r)}{dr^2}+\frac{2}{r}\frac{d a_{l,m}(r)}{dr}+(k^2-\frac{l(l+1)}{r^2})a_{l,m}(r)=0,
\end{eqnarray}
which is solved by spherical Bessel functions and spherical Neumann functions.
Let
\begin{equation}\label{118}
y_{l,m}(r):=ra_{l,m}(r),
\end{equation}
so we obtain
\begin{eqnarray}\label{1.10}
\left\{%
\begin{array}{ll}
y_{l,m}''(r)+(k^2-\frac{l(l+1)}{r^2})y_{l,m}(r)=0;\vspace{5pt}\\
 y_{l,m}(0)=0.
\end{array}%
\right.
\end{eqnarray}
\par
To give an initial condition, we apply the boundary conditions in~(\ref{1.1}) to $u_{l,m}(x;k)$ near the intersection points $\hat{R}$ along $\hat{x}_0$.
We replace the boundary condition $\frac{\partial u}{\partial \nu}=0$ to be $\nabla u=0$.
Hence,
\begin{eqnarray}\label{1.11}
&&[\frac{y_{l,m}(r;k)}{r}]'|_{r=\hat{R}}=0;\\
&&[\frac{y_{l,m}(r;k)}{r}]|_{r=\hat{R}}=1,\label{1.12}
\end{eqnarray}
in which we assume there is no tangent point.
The solutions $y_{l,m}(r;k)$ are independent of $m$, so we write $y_{l,m}(r;k)$ as $y_{l}(r;k)$.
Hence, now we have following boundary conditions.
\begin{eqnarray}\label{1.13}
&&\frac{y_l'(\hat{R};k)}{\hat{R}}-\frac{y_l(\hat{R};k)}{\hat{R}^2}=0;\\
&&y_l(\hat{R};k)- \hat{R}=0.\label{1.14}
\end{eqnarray}
If $k$ satisfies~(\ref{1.13}) and~(\ref{1.14}), then $y_l'(\hat{R};k)=1$. Thus,~(\ref{1.11}) and~(\ref{1.12}) equivalently satisfy
\begin{eqnarray}\label{113}
&&F_l(k;\hat{R}):=y_l(\hat{R};k)- \hat{R}=0;\\
&&G_l(k;\hat{R}):=y_l'(\hat{R};k)-1=0.\label{114}
\end{eqnarray}
In the initial state, $y_l(\hat{R};k)$ is exactly the boundary defining function of $D$.
Combining~(\ref{1.10}),~(\ref{113}), and~(\ref{114}),
we consider the following eigenvalue problem at $\hat{R}$ for each fixed $\hat{x}\in\mathbb{S}^2$ and all $l\geq0$:
\begin{eqnarray}\label{15}
\left\{%
\begin{array}{ll}
   y_l''(r;k)+(k^2-\frac{l(l+1)}{r^2})y_l(r;k)=0,\,0<r<\infty;\vspace{4pt}\\
   \vspace{3pt}
    y_{l}(0;k)=0;\\ \vspace{3pt}
   F_l(k;\hat{R})=0;\\ \vspace{3pt}
    G_l(k;\hat{R})=0.
\end{array}%
\right.
\end{eqnarray}
This is a two-way initial value problem starting at $r=\hat{R}$ inward and outward. The eigenvalue $k$ passes through to the infinity by the uniqueness of the ODE and defines the far-field patterns near infinity. There is an one-to-one correspondence between the far-field patterns and the radiating solution of the Helmholtz equation. The $y_l(r;k)$ depends on the incident angle $\hat{x}$. Most important of all, we will examine the zero set of  $y_{l}(0;k)=0$ which constitutes the eigenvalues of~(\ref{15}).
The solutions $\{y_l(r;k)\}_{l\geq0}$ is a family of entire functions of exponential type \cite{Carlson,Carlson2,Carlson3,Po}.  For each $l\geq0$, it behaves like sine functions in complex plane with zero set asymptotically approaching the zero set of sine functions for each incident direction.
The Weyl's law of the eigenvalues of~(\ref{15}) in many settings are found in \cite{Chen,Chen3,Chen5} as a direct consequence of the Cartwright-Levinson theory in value distribution theory \cite{Boas,Cartwright,Cartwright2,Koosis,Levin,Levin2}. In particular,
 we can find that the density of the zero set for each incident direction is related to the radius $\hat{R}$ as a spectral invariant. Rellich's lemma indicates that all perturbations behave like spherical waves near the infinity, by which we prove \textbf{a special case} of Schiffer's conjecture.
\begin{theorem}\label{11}
Let $D$ be a starlike domain as assumed in~(\ref{1.1}) under radiation condition~(\ref{122}). If there is an eigenvalue $k_0^2\in\mathbb{R}^+$, $k_0^2\geq1$ of~(\ref{1.1}), then $D$ is an open ball centered at the origin.
\end{theorem}

\section{Singular Sturm-Liouville Theory}

Here we collect the asymptotic behaviors for $y_l(r;k)$ and $y_l'(r;k)$. For $l\geq0$, we apply the results from \cite{Carlson,Carlson2,Carlson3,Po}. Let $z_l(\xi;k)$ be the solution of
\begin{eqnarray}\label{21}
\left\{
  \begin{array}{ll}
    -z_l''(\xi)+\frac{l(l+1)z_l(\xi)}{\xi^2}+p(\xi)z_l(\xi)=k^2z_l(\xi);\vspace{9pt}\\
    z_l(1;k)=-b;\,z_l'(1;k)=a,\,a,\,b\in\mathbb{R},
  \end{array}
\right.
\end{eqnarray}
where $p(\xi)$ is square integrable; the real number $l\geq-1/2$.
In general,
\begin{eqnarray}\label{123}
|z_l(\xi;k)+b\cos{k(1-\xi)}+a\frac{\sin{k(1-\xi)}}{k}|\leq \frac{K(\xi)}{|k|}\exp\{|\Im k|[1-\xi]\},\,|k|\geq1,
\end{eqnarray}
where
\begin{equation}\label{124}
K(\xi)\leq\exp\{\int_\xi^1\frac{|l(l+1)|}{t^2}+|p(t)|dt\},\,0\leq\xi\leq 1.
\end{equation}
This explains the behaviors of solutions $z_l(\xi;k)$ and $z_l'(\xi;k)$ for all $l$ in unit interval. For its application in~(\ref{113}) and~(\ref{114}), we take $$b=-\hat{R};\,a=1$$ for each incident direction, and the problem~(\ref{21}) in interval $[0,\hat{R}]$.
\par
Outside the domain $D$, we consider~(\ref{15}) as an initial problem starting at $\hat{R}$ to the infinity. If $p(\xi)\equiv0$, then we consider the following special case:
\begin{eqnarray}\label{2.4}
v_l''(\xi)+[k^2-\frac{l(l+1)}{\xi^2}]v_l(\xi)=0.
\end{eqnarray}
The solutions of~(\ref{2.4}) are essentially Bessel's functions with a basis of two elements.
The variation of parameters formula leads to the following asymptotic expansions: For $\xi>0$ and $\Re k\geq0$, there is a constant $C$ so that
\begin{eqnarray}
&&|v_l(\xi,k)-\frac{\sin\{k\xi-l\frac{\pi}{2}\}}{k^{l+1}}|\leq C|k|^{-(l+1)}\frac{\exp\{|\Im k|\xi\}}{|k\xi|};\label{2.5}\\
&&|v_l'(\xi,k)-\frac{\cos\{k\xi-l\frac{\pi}{2}\}}{k^{l}}|\leq C|k|^{-l}\frac{\exp\{|\Im k|\xi\}}{|k\xi|}.\label{2.6}
\end{eqnarray}
We refer these estimates to \cite[Lemma\,3.2,\,Lemma\,3.3]{Carlson2}, and the we find that a solution of the initial value problem of~(\ref{2.4}) is a linear combination of~(\ref{2.5}) and~(\ref{2.6}).

\section{Cartwright-Levinson Theory}

We take the following vocabularies from entire function theory \cite{Boas,Cartwright,Cartwright2,Koosis,Levin,Levin2} to describe the asymptotic behavior of the eigenvalues of~(\ref{15}).
\begin{definition}
Let $f(z)$ be an integral function of order $\rho$,
$N(f,\alpha,\beta,r)$ be the number of the zeros of $f(z)$
inside the angle $[\alpha,\beta]$, and $|z|\leq r$. We define the
density function as
\begin{equation}\label{Den}
\Delta_f(\alpha,\beta):=\lim_{r\rightarrow\infty}\frac{N(f,\alpha,\beta,r)}{r^{\rho}},
\end{equation}
and
\begin{equation}
\Delta_f(\beta):=\Delta_f(\alpha_0,\beta),
\end{equation}
with some fixed $\alpha_0\notin E$, in which $E$ is at most a
countable set \cite{Levin,Levin2}.
\end{definition}
Let us define
\begin{equation}
\hat{\Delta}(\xi):=\Delta_{z_l(\xi;k)}(-\epsilon,\epsilon),\,b=-\hat{R},
\end{equation}
as the density of the zero set along $\hat{x}$.
\begin{lemma}
The entire functions $y_l(\xi;k)$ and $y_l'(\xi;k)$ are of order one and of type $\hat{R}-\xi$.
\end{lemma}
\begin{proof}
From~(\ref{123}), we have
\begin{equation}\label{3.4}
y_l(\xi;k)=-\hat{R}\cos{k(\hat{R}-\xi)}-\frac{\sin{k(\hat{R}-\xi)}}{k}+O( \frac{K(\xi)}{|k|}\exp\{|\Im k|[\hat{R}-\xi]\}),\,|k|\geq1.
\end{equation}
To find the type of an entire function, we compute the following definition of Lindel\"{o}f's indicator function \cite{Levin,Levin2}
\begin{definition}
Let $f(z)$ be an integral function of finite order $\rho$ in the
angle $[\theta_1,\theta_2]$. We call the following quantity as the
indicator of the function $f(z)$.
\begin{equation}
h_f(\theta):=\lim_{r\rightarrow\infty}\frac{\ln|f(re^{i\theta})|}{r^{\rho}},
\,\theta_1\leq\theta\leq\theta_2.
\end{equation}
\end{definition}
We find that  if $k=|k|e^{i\theta}$, then
\begin{equation}\label{3.6}
h_{y_l(\xi;k)}(\theta)=|(\hat{R}-\xi)\sin\theta|,\,\theta\in[0,2\pi],
\,0<\xi<\hat{R}.
\end{equation}
When referring more details to \cite{Chen,Chen3,Chen5,Cartwright2,Levin,Levin2}, we find more examples  in \cite[p.\,70]{Cartwright2}. The maximal value of $h_{y_l(\xi;k)}(\theta)$ gives the type of an entire function \cite[p.\,72]{Levin}, which is $(\hat{R}-\xi)$. A similar proof holds for $y_l'(\xi;k)$.
\end{proof}
More importantly, the indicator function~(\ref{3.6}) leads to the following Cartwright's theory \cite[p.\,251]{Levin}.
\begin{lemma}\label{34}
We have the following asymptotic behavior of the zero set of $y_l(\xi;k)$.
$$\hat{\Delta}(\xi)=\frac{\hat{R}-\xi}{\pi}.$$
\end{lemma}
\begin{proof}
We observe in~(\ref{3.4}) that $|y_l(\xi;k)|$ is bounded on the real axis.
Hence, it is in Cartwright's class. All of the properties in  \cite[p.\,251]{Levin} hold.
\end{proof}
Letting $\xi=0$, we obtain the eigenvalue density of~(\ref{15}) in $\mathbb{C}$. Moreover, they are all real.
\begin{lemma}
The eigenvalues $k$ of~(\ref{15}) are all real.
\end{lemma}
\begin{proof}
For $l=0$, the result is classic \cite{Carlson2,Po}. In our case, $y_l(\xi;k)$ is real for $k\in0i+\mathbb{R}$. Furthermore, the asymptotic behavior of~(\ref{3.4}) proves the lemma, which is a special case of Bernstein's theorem in entire function theory \cite[Theorem 1]{Duffin}. A step-by-step proof is provided in \cite[Lemma 2.6]{Chen5}.
\end{proof}

\section{Proof of Theorem \ref{11}}
\begin{proof}
Let $k_0^2$ be an eigenvalue of~(\ref{1.1}), as assumed in Theorem \ref{11}.
Particularly, from~(\ref{1.2}) we have
\begin{eqnarray}\label{4.1}
&&u(x;k_0)=\sum_{l=0}^{\infty}\sum_{m=-l}^{m=l}
a_{l,m}(r;k_0)Y_l^m(\hat{x});\\
&&u_{l,m}(x;k_0)=a_{l,m}(r;k_0)Y_l^m(\hat{x}),
\,\hat{x}\in\mathbb{S}^2,\label{4.2}
\end{eqnarray}
in which the coefficient $a_{l,m}(r;k_0)$ does not depend on the incident direction $\hat{x}\in\mathbb{S}^2$ for sufficiently large $|x|:=r$:  The functions in~(\ref{4.1}) solve the Helmholtz equation in $r\geq R_0$. As a result of the uniqueness of the ODE~(\ref{15}), the solutions $y_l(r;k_0)$ extend both outward to the infinity and inward to the origin for all $l\geq0$.
For the given eigenvalue $k_0^2$, the equation~(\ref{15}) holds for all incident directions $\hat{x}\in\mathbb{S}^2$ and for all $l\geq0$.

\par
The representation in~(\ref{4.1}) is unique in $\mathbb{R}^3$: If there is another eigenvalue $k'$ of~(\ref{15}) from incident angle $x'\neq \hat{x}\in\mathbb{S}^2$ with the solution
$$u_{l,m}'(x;k'):=a_{l,m}'(r;k')Y_l^m(\hat{x}),$$
then the analytic continuation of Helmholtz equation \cite[p.\,18]{Colton2} implies that
\begin{equation}\label{43}
a_{l,m}'(r;k')=a_{l,m}(r;k_0).
\end{equation}
With the uniqueness of the ODE~(\ref{1.10}), $k_0$ or $k'$ satisfies~(\ref{15}) individually along its own incident direction inward to the origin. Therefore, Lemma \ref{34} provides an eigenvalue density $$\hat{\Delta}(0)=\frac{\hat{R}}{\pi},\,\hat{x}\in\mathbb{S}^2.$$

\par
The ODE~(\ref{15}) holds for all $\xi\geq\hat{R}$ and $l\geq 0$. In particular, we apply the estimates
~(\ref{2.5}) and~(\ref{2.6}):
\begin{eqnarray*}
&&|v_l(\xi,k)-\frac{\sin\{k(\xi-\hat{R})-l\frac{\pi}{2}\}}{k^{l+1}}|\leq C|k|^{-(l+1)}\frac{\exp\{|\Im k|\xi\}}{|k\xi|};\\
&&|v_l'(\xi,k)-\frac{\cos\{k(\xi-\hat{R})-l\frac{\pi}{2}\}}{k^{l}}|\leq C|k|^{-l}\frac{\exp\{|\Im k|\xi\}}{|k\xi|}.
\end{eqnarray*}
Therefore, the initial value problem~(\ref{21}), with $p\equiv 0$, $b=\hat{R}$, and $a=1$, provides the asymptotic behavior for the solution:
\begin{equation}\nonumber
y_l(\xi;k)=\hat{R}\frac{\cos\{k(\xi-\hat{R})-l\frac{\pi}{2}\}}{k^l}
+O\{\frac{\exp\{|\Im k|\xi\}}{k^{l}(k\xi)}\}.
\end{equation}
That is,
\begin{equation}\label{4.6}
k^ly_l(\xi;k)=\hat{R}\cos\{k(\xi-\hat{R})-l\frac{\pi}{2}\}
[1+O\{\frac{1}{k\xi}\}],\,\hat{R}<\xi<\infty,
\end{equation}
outside the zeros of $\cos\{k(\xi-\hat{R})-l\frac{\pi}{2}\}$. This is classic in Sturm-Liouville theory \cite{Carlson,Carlson2,Po}.

\par
The given eigenvalue $k_0$ satisfies~(\ref{15}), for all $l\geq0$ and all $\hat{x}\in\mathbb{S}^2$, and~(\ref{4.6}). Therefore,
\begin{equation}
k_0^ly_l(\xi;k_0)=\hat{R}\cos\{k_0(\xi-\hat{R})-l\frac{\pi}{2}\}
[1+O\{\frac{1}{k_0\xi}\}],\,0<\xi<\infty.
\end{equation}
We choose $l\uparrow\infty$ and so $\xi\uparrow\infty$ such that $\xi=\hat{R}+\frac{l\pi}{2k_0}>R_0$ for any large $R_0$.
Thus, for large $l$,
\begin{equation}
k_0^ly_l(\xi;k_0)=\hat{R}+O(\frac{1}{\xi}),\,
\xi=\hat{R}+\frac{l\pi}{2k_0},\,|k_0|\geq1.
\end{equation}
Using~(\ref{118}),~(\ref{1.2}) and the uniqueness of the Helmholtz equation,
as shown in~(\ref{43}), the far-field patterns \cite[(2.49)]{Colton2} are asymptotically the same periodic functions for each $\hat{x}\in\mathbb{S}^2$. In particular, the boundary defining function $\hat{R}$ is constant to $\hat{x}\in\mathbb{S}^2$, and Theorem \ref{11} is thus proven.
\end{proof}

\begin{compliance}
The author declares there is no conflicts of interest regarding the publication of this paper.
The research does not involve any human participant and/or animals, and no further informed consent is required.
\end{compliance}

\end{document}